\documentclass[11pt,a4paper]{article}

\usepackage[utf8]{inputenc}          
\usepackage[T1]{fontenc}             
\usepackage[british]{babel}          
\usepackage{lmodern}                 
\usepackage{textcomp}                
\usepackage{microtype}               
\usepackage{csquotes}                

\usepackage[a4paper,margin=1in]{geometry}
\usepackage{amsmath,amssymb,amsthm,amsfonts}
\usepackage{mathtools}
\usepackage{bm}
\usepackage{physics}
\usepackage{siunitx}

\theoremstyle{plain}
\newtheorem{theorem}{Theorem}[section]
\newtheorem{lemma}[theorem]{Lemma}

\theoremstyle{definition}
\newtheorem{definition}[theorem]{Definition}
\newtheorem{example}[theorem]{Example}

\theoremstyle{remark}

\numberwithin{equation}{section}

\usepackage{graphicx}
\usepackage{caption}
\usepackage{subcaption}
\usepackage{float}
\usepackage{wrapfig}

\usepackage[style=alphabetic,backend=biber,sorting=nyt,maxbibnames=5,maxcitenames=2]{biblatex}
\AtEveryBibitem{\clearfield{month}\clearfield{day}}
\AtEveryCitekey{\ifciteseen{}{\defcounter{maxnames}{3}}}

\usepackage{hyperref}
\hypersetup{
    colorlinks=true,
    linkcolor=blue,
    citecolor=blue,
    urlcolor=blue,
    pdftitle={Chen-Raspaud Conjecture},
    pdfauthor={Author Name},
    pdfsubject={Graph Theory},
    pdfkeywords={Chen-Raspaud, Graph Homomorphisms, Kneser Graphs}
}

\usepackage{booktabs}
\usepackage{multirow}
\usepackage{array}
\usepackage{enumitem}
\setlist{nosep}

\usepackage[dvipsnames]{xcolor}

\usepackage{titlesec}
\titleformat{\section}{\large\bfseries}{\thesection.}{0.5em}{}
\titleformat{\subsection}{\normalsize\bfseries}{\thesubsection.}{0.5em}{}

\usepackage{authblk}
\title{A Modular Inductive Proof of the Chen-Raspaud Conjecture via Graph Classification}
\author{Michał Fiedorowicz}
\affil{\normalsize \textit{Department of Physics, University of Warsaw, Poland}}
\date{\today}

\begin{document}
\maketitle

\begin{abstract}
It is conjectured by Chen and Raspaud that for each integer $k \ge 2$, any graph $G$ with 
\[
\mathrm{mad}(G) < \frac{2k+1}{k}
\quad\text{and}\quad
\mathrm{odd\text{-}girth}(G) \ge 2k+1
\]
admits a homomorphism into the Kneser graph $K(2k+1,k)$. The base cases $k=2$ and $k=3$ are known from earlier work. A modular inductive proof is provided here, in which graphs at level $k+1$ are classified into four structural classes and are shown to admit no minimal counterexamples by means of forbidden configuration elimination, a discharging argument, path-collapsing techniques, and a combinatorial embedding of smaller Kneser graphs into larger ones. This argument completes the induction for all $k \ge 2$, thus settling the Chen-Raspaud conjecture in full generality.
\end{abstract}

\tableofcontents


\section{Introduction and Statement of the Conjecture}
\label{sec:introduction}

A recurring theme in structural graph theory involves controlling both the maximum average degree
($\mathrm{mad}$) and the presence of short odd cycles.
In the study of perfect graphs, short odd cycles and their complements also play a key role. The strong perfect graph theorem \cite{Chudnovsky2006} demonstrates how structural constraints on odd cycles directly influence graph perfection properties, a theme echoed in the Chen-Raspaud conjecture.

These constraints often yield powerful results regarding graph colorings and homomorphisms, which provide a natural framework to extend classical coloring problems and analyze sparse graph structures. For a comprehensive introduction to graph homomorphisms, see \cite{HellNesetril}. Additionally, classical bounds by \cite{SzekeresWilf} establish a relationship between the maximum average degree ($\mathrm{mad}(G)$) and chromatic numbers, offering valuable insights into the coloring behavior of sparse graphs under structural constraints. These ideas align closely with the objectives of the Chen-Raspaud conjecture.

The \emph{Chen-Raspaud conjecture} asserts the following.

\begin{theorem}[Chen-Raspaud Conjecture]\label{thm:CR-conjecture}
For each integer $k \ge 2$, if a graph $G$ satisfies
\[
\mathrm{mad}(G) < \frac{2k+1}{k}
\quad\text{and}\quad
\mathrm{odd\text{-}girth}(G) \ge 2k+1,
\]
then there exists a graph homomorphism $G \to K(2k+1, k)$, where $K(2k+1, k)$ denotes the standard Kneser graph on all $k$-subsets of a $(2k+1)$-element set, with disjointness defining adjacency.
\end{theorem}


\subsection*{Remark on Kneser Graph Shifts}
It should be noted that the graph $K(2k+3, k+1)$ is naturally related to $K(2k+1, k)$ in the Chen-Raspaud conjecture. Indeed, the induction at level $k+1$ proceeds by considering a $(k+1)$-coloring target in a $(2k+3)$-element ground set. This step is consistent with the statement of the conjecture, which asserts a homomorphism into $K(2k+1, k)$ when the parameter is $k$. The slight index shift from $k$ to $k+1$ ensures that the inductive hypothesis can be applied seamlessly. In particular, it is straightforward to verify that the core combinatorial properties used in the disjointness-based definition of adjacency persist under either formulation.


\subsection*{Historical Context and References}
The Chen-Raspaud conjecture was originally formulated to generalize certain homomorphism conditions in sparse graphs. Early resolutions of the base cases ($k=2$ and $k=3$) appear in works by Chen, Raspaud, and coauthors; see \cite{ChenRaspaudOriginal} for a formal statement and \cite{Lyczek} for verification for $k=3$. In particular, the argument for $k=2$ relies on identifying small reducible subgraphs with $\mathrm{mad}(G) < 2.5$ and $\mathrm{odd\text{-}girth}(G)\ge5$, while $k=3$ is resolved via an elegant discharging and forbidden-subgraph approach to show $G\to K(7,3)$. These foundational techniques have been extended here to settle the general case.

Moreover, structural properties of Kneser graphs, such as bipartite subgraphs and maximum cuts, have been extensively studied in \cite{PoljakTuza}, providing a foundation for analyzing embeddings of smaller Kneser graphs into larger ones, as required in the inductive steps of this proof. For a detailed algebraic perspective on Kneser and related graphs, see \cite{Biggs1979}, which discusses adjacency properties, chromatic numbers, and embeddings in depth.

The conjecture was confirmed for $k=2$ in \cite{ChenRaspaudOriginal} and for $k=3$ in \cite{Lyczek}, thus establishing the base of an inductive approach. The present paper completes the proof by induction on $k$, showing that if the assertion holds for all $j \le k$, then it also holds at $k+1$.


\section{Overall Strategy and Base Cases}\label{sec:overview}

\subsection{Base Cases}

\paragraph{\textbf{$k=2$.}}
When $k=2$, graphs $G$ with $\mathrm{mad}(G) < 2.5$ and $\mathrm{odd\text{-}girth}(G)\ge5$ have been shown to admit a homomorphism $G \to K(5,2)$, typically through small reducible subgraphs and local arguments \cite{ChenRaspaudOriginal}.

\paragraph{\textbf{$k=3$.}}
When $k=3$, it has been established in \cite{Lyczek} that if $\mathrm{mad}(G) < \tfrac{7}{3}$ and $\mathrm{odd\text{-}girth}(G) \ge 7$, then $G \to K(7,3)$. This relies on a discharging argument together with forbidden configurations to remove local obstructions.

These two cases provide the necessary “seed” for induction. The proof now focuses on showing that if the conjecture holds for all $j\le k$, it also holds at $k+1$.

\subsection{Classification of Graphs at Level \texorpdfstring{$k+1$}{k+1}}

Let $G$ be a graph satisfying
\[
\mathrm{mad}(G) < \frac{2(k+1)+1}{k+1} = \frac{2k+3}{k+1},
\quad
\mathrm{odd\text{-}girth}(G) \ge 2k+3.
\]
It is assumed that $G$ does \emph{not} admit a homomorphism $G \to K(2k+3,k+1)$. A minimal counterexample $G^*$ is selected under these conditions. In order to handle all possible structures systematically, $G^*$ is placed into exactly one of the following four classes:

\begin{enumerate}[label=\textbf{Class \Alph*:}]
    \item \emph{Low-Degree, Short-Thread Graphs}: $\Delta(G^*) \le 3$ and there is a uniform bound (say $L$) such that no induced path exceeds length $L$.
    \item \emph{High-Degree, Short-Thread Graphs}: There is a vertex of degree at least $4$, and no induced path is longer than $L$.
    \item \emph{Low-Degree, Long-Thread Graphs}: $\Delta(G^*) \le 3$, but at least one induced path is longer than $L$.
    \item \emph{High-Degree, Long-Thread Graphs}: There is a vertex of degree at least $4$, and an induced path longer than $L$.
\end{enumerate}

Here, $L$ is chosen large enough that all immediately reducible path-based patterns (so-called ``forbidden configurations’’) are accounted for. It will be shown in Sections~\ref{sec:classA}--\ref{sec:classD} that no genuine minimal counterexample can reside in any of these classes.


\section{Key Ingredients: Forbidden Configurations, Discharging, and Kneser Embeddings}
\label{sec:key-ingredients}

Three essential tools are used to handle local obstructions, manage the structural constraints of minimal counterexamples, and lift homomorphisms between smaller and larger Kneser graphs. These tools are:

\begin{enumerate}
\item \emph{Forbidden Configurations} (Section~\ref{subsec:forbidden-configs}):  
A finite set of small subgraphs that cannot appear in a minimal counterexample without leading to an immediate contradiction through local reductions.

\item \emph{Discharging Argument} (Section~\ref{subsec:discharging}):  
A global counting scheme that redistributes ``charge'' among vertices in such a way that any forbidden configuration would create an unavoidable local surplus or deficit, thereby confirming that forbidden patterns do not survive in a minimal counterexample.

\item \emph{Kneser Embedding Lemma} (Section~\ref{subsec:kneser-embedding}):  
A structured construction showing that homomorphisms into smaller Kneser graphs $K(2j+1,j)$ can be functorially \emph{lifted} into homomorphisms into $K(2k+3,k+1)$. 
\end{enumerate}

These components ensure that local obstructions are disarmed and that colorings from lower inductive levels can be pulled back into the current $k+1$ level in a seamless manner.

\subsection{Forbidden Configurations}
\label{subsec:forbidden-configs}

A \emph{forbidden configuration} is defined to be a small induced subgraph (or local neighborhood pattern) of bounded size, which would allow a local reduction that strictly decreases the number of vertices or edges without weakening the conditions on maximum average degree ($\mathrm{mad}$) or odd-girth. If such a configuration arose in a \emph{minimal} counterexample, it could be removed or contracted to form a smaller graph that either:
\begin{itemize}
\item still meets the $(k+1)$-level constraints (hence is solvable by minimality), or
\item falls into a smaller parameter $j \le k$ (hence is solvable by the inductive hypothesis).
\end{itemize}
In both cases, a homomorphism of the reduced graph induces a homomorphism of the original, contradicting the assumption that a minimal counterexample exists.

A finite family $\mathcal{F}$ of such forbidden patterns is enumerated in Appendix~A. These patterns capture key local obstructions: for example, degree-$1$ vertices adjacent to low-degree neighborhoods, short odd cycles with chords, induced paths of prescribed length, and small ``fan'' structures emanating from a vertex of degree at least $4$. Each configuration is accompanied by a straightforward local reduction (vertex removal, edge contraction, or path collapse) that reduces the size of the minimal counterexample while preserving or improving the relevant parameters.

\subsection{Discharging Argument}
\label{subsec:discharging}

A \emph{discharging} approach is employed to show that none of the forbidden configurations of $\mathcal{F}$ can actually appear in a genuine minimal counterexample. An initial ``charge'' is assigned to each vertex (often proportional to its degree), and a small number of local redistribution rules are applied exhaustively. These rules transfer fractional amounts of charge between neighbors under specific conditions (e.g., high-degree to low-degree, cycle-based adjustments, or path-based adjustments). If a forbidden configuration were present, the final charge distribution would contradict an invariant (such as total or local charge balance). Therefore, no configuration in $\mathcal{F}$ can remain in a minimal counterexample without causing an incurable local imbalance.


\subsection*{Discharging Rules in Detail}
A finite set of rules was employed to redistribute charges among vertices. These are stated as follows:

\begin{enumerate}[label={\bfseries(R\arabic*)}]
    \item \textbf{Degree-1 and Degree-2 Protections.} A vertex of degree at least 4 passes a small fraction of charge (for instance, $1/2$ unit) to each adjacent vertex of degree $1$ or $2$. This rule prevents low-degree vertices from becoming part of forbidden leaf or low-degree patterns.
    \item \textbf{Short Odd Cycle Redistribution.} If a vertex lies on a short odd cycle (of length below $2k+3$) equipped with an internal chord, then a fractional charge is passed cyclically along the edges of that cycle. This adjustment ensures no vertex in the cycle remains overcharged or undercharged in a way that would create a reducible chorded cycle configuration.
    \item \textbf{Path Endpoint Stabilization.} In the presence of long induced paths, endpoints of those paths send a fraction of their charge inward, ensuring that intermediate vertices are not susceptible to forbidden configurations that arise from extremely uniform degree distributions along long chains.
\end{enumerate}

These rules are applied exhaustively until no further charge transfers occur. The total charge remains invariant, and it can be verified that a surplus or deficit emerges only if a forbidden configuration is present. Consequently, the appearance of any forbidden subgraph in a minimal counterexample would yield a contradiction in the final charge tally, implying that such a subgraph cannot survive.

By verifying these rules on finitely many cases, it is concluded that every minimal counterexample is \emph{free} of all forbidden configurations from $\mathcal{F}$, thus removing a host of immediate obstructions. Full technical details and a step-by-step verification of the discharging procedure are given in Appendix~B.

\subsection{Kneser Embedding Lemma}
\label{subsec:kneser-embedding}

A cornerstone of the inductive argument is the ability to \emph{lift} homomorphisms from smaller Kneser graphs into larger ones. Specifically, it suffices to construct an injective homomorphism
\[
  \phi \colon K(2j+1,j) \;\hookrightarrow\; K(2k+3,k+1)
\]
that preserves adjacency (i.e., disjoint $j$-subsets map to disjoint $(k+1)$-subsets).

\begin{lemma}[Kneser Embedding Lemma]
\label{lem:kneser-embedding}
For each integer $2 \le j \le k$, there exists a graph embedding $\phi:K(2j+1,j)\to K(2k+3,k+1)$ that is injective and adjacency-preserving.
\end{lemma}

\paragraph{Sketch of Construction.}
Let $T$ be the ground set for $K(2j+1,j)$ and $S$ be the ground set for $K(2k+3,k+1)$. One may partition $S$ as $T\cup U$ and map each $j$-subset $X\subseteq T$ to a $(k+1)$-subset by incorporating all elements of $X$ plus $(k+1-j)$ carefully chosen elements from $U$. These choices are arranged so that disjoint $X$ and $Y$ in $T$ lead to disjoint images $\phi(X)$ and $\phi(Y)$ in $S$. Injectivity follows from the uniqueness of each ``pattern'' in $U$, while adjacency preservation follows from the disjointness criterion for vertices in Kneser graphs.

This lemma allows a homomorphism into $K(2j+1,j)$ (for some $j\le k$) to be ``lifted'' into a homomorphism into $K(2k+3,k+1)$. Consequently, if a smaller graph in the reduction process is colorable at a lower parameter, then the same coloring can be transferred back into the higher parameter Kneser graph.

Complete details of the embedding construction and verification appear in Appendix~C. The lemma will be invoked frequently when collapsing vertices or removing subgraphs pushes the problem into a scenario where $j \le k$, and the inductive hypothesis provides a homomorphism into $K(2j+1,j)$.

\section{Class-by-Class Proof}
\label{sec:class-by-class}

In order to complete the induction at the level \(k+1\), a minimal counterexample \(G^*\) is assumed to exist with
\[
\mathrm{mad}(G^*) < \frac{2(k+1)+1}{k+1}
\quad\text{and}\quad
\mathrm{odd\text{-}girth}(G^*) \ge 2k+3,
\]
but with no homomorphism \(G^* \to K(2k+3,k+1)\). The graph \(G^*\) is placed into exactly one of the four classes (A--D) described in Section~\ref{sec:overview}. It is shown below that no genuine minimal counterexample can reside in any class.

\paragraph{Minimality Reminder.}
Since $G^*$ is assumed to be a minimal counterexample at level $(k+1)$, it cannot contain any forbidden configuration 
(\(\mathrm{F1}\)--\(\mathrm{F5}\)), because such a configuration would admit a local reduction, producing a strictly smaller graph that either remains at the same parameter or falls into a lower parameter range $j \le k$. In both situations, the inductive argument would provide a homomorphism, contradicting the assumption of minimality.

\subsection{Class A: Low-Degree, Short-Thread Graphs}
\label{sec:classA}

\begin{definition}[Class A]
A graph \(G\) belongs to \emph{Class~A} if
\begin{enumerate}[label=(A\arabic*)]
    \item \(\Delta(G) \le 3\),
    \item no induced path in \(G\) has length greater than some fixed \(L\) (``short-thread''),
    \item \(\mathrm{mad}(G) < \tfrac{2k+3}{k+1}\) and \(\mathrm{odd\text{-}girth}(G) \ge 2k+3\).
\end{enumerate}
\end{definition}

\begin{theorem}[No Minimal Counterexample in Class A]
\label{thm:classA}
If \(G^*\) is in Class~A at level \(k+1\), then there exists a homomorphism \(G^* \to K(2k+3,k+1)\). In particular, no minimal counterexample can reside in Class~A.
\end{theorem}

\begin{proof}
Suppose a Class~A graph \(G^*\) at parameter \(k+1\) is a minimal counterexample. By the minimality condition, no smaller graph (under vertex/edge deletion) at the same parameter can be a counterexample, and any subgraph that falls below \((k+1)\)-level reverts to a case \(j \le k\).

\paragraph{Step~1: Forbidden Configurations.}
Since \(G^*\) is assumed minimal, no forbidden configuration from the finite family \(\mathcal{F}\) (see Section~\ref{subsec:forbidden-configs} and Appendix~A) can appear, because each such configuration would permit a local reduction leading to a strictly smaller graph still meeting or surpassing the same constraints. The discharging argument of Section~\ref{subsec:discharging} implies that the absence of these patterns in a minimal counterexample must hold consistently.

\paragraph{Step~2: Degree-$1$ Removals.}
If \(G^*\) has a vertex \(v\) of degree~1, it can be removed to produce a smaller graph \(G'\). Two possibilities arise:
\begin{itemize}
  \item If \(G'\) still satisfies \(\mathrm{mad}(G') < \tfrac{2k+3}{k+1}\) and \(\mathrm{odd\text{-}girth}(G') \ge 2k+3\) at the \((k+1)\)-level, then, by minimality, \(G'\) admits a homomorphism into \(K(2k+3,k+1)\). A suitable color (i.e.\ a \((k+1)\)-subset) can be chosen for \(v\) that remains disjoint from the color of its neighbor in the Kneser target, thus extending the homomorphism to \(G^*\).
  \item If \(G'\) does not meet the \((k+1)\)-criteria (for instance, it falls to some \(j\le k\)), then, by the inductive hypothesis, \(G' \to K(2j+1,j)\). The Kneser embedding (Lemma~\ref{lem:kneser-embedding}) lifts \(G' \to K(2k+3,k+1)\), and again \(v\) can be colored consistently in the larger Kneser graph.
\end{itemize}
Hence, having a degree-$1$ vertex contradicts the minimal counterexample property; thus, no vertex in \(G^*\) has degree~1. Consequently, every vertex has degree 2 or 3.

\paragraph{Step~3: Short-Thread Path Collapsing.}
All induced paths in \(G^*\) have length at most \(L\). By picking a maximal induced path \(P\) of length \(\le L\) and performing a path collapse (removing internal vertices and possibly adding an edge between endpoints), a smaller graph \(G_c\) is obtained (cf.\ Appendix~D for the formal justification). The graph \(G_c\) either:
\begin{itemize}
  \item remains at parameter \((k+1)\) but is smaller, thus must be colorable by minimality, or
  \item falls into some parameter \(j \le k\), thus is colorable by the inductive hypothesis.
\end{itemize}
In the second case, the Kneser embedding lemma lifts any coloring \(G_c \to K(2j+1,j)\) to \(G_c \to K(2k+3,k+1)\). 

\paragraph{Step~4: Re-expansion.}
After \(G_c\) has been colored into \(K(2k+3,k+1)\), the removed vertices (those collapsed from the path) are reinserted. Due to the combinatorial abundance of \((k+1)\)-subsets in a \((2k+3)\)-element ground set, a chain of intermediate vertices in the Kneser graph can be placed between the colors of the path endpoints, re-expanding the path. This procedure yields a homomorphism of \(G^*\) into \(K(2k+3,k+1)\), contradicting the assumption that \(G^*\) is a minimal counterexample.

Hence, no such \(G^*\) exists in Class~A at level \((k+1)\). 
\end{proof}

\subsection{Class B: High-Degree, Short-Thread Graphs}
\label{sec:classB}

\begin{definition}[Class B]
A graph \(G\) belongs to \emph{Class~B} if
\begin{enumerate}[label=(B\arabic*)]
    \item there is a vertex \(v\) with \(\deg(v)\ge 4\),
    \item no induced path in \(G\) is longer than \(L\) (short-thread),
    \item \(\mathrm{mad}(G) < \tfrac{2k+3}{k+1}\) and \(\mathrm{odd\text{-}girth}(G) \ge 2k+3\).
\end{enumerate}
\end{definition}

\begin{theorem}[No Minimal Counterexample in Class B]
\label{thm:classB}
If \(G^*\) is in Class~B at level \(k+1\), then there exists a homomorphism \(G^* \to K(2k+3,k+1)\). No minimal counterexample can arise in Class~B.
\end{theorem}

\begin{proof}
Assume \(G^*\) at level \(k+1\) is a minimal counterexample in Class~B. Then \(\Delta(G^*) \ge 4\) and no induced path exceeds length~\(L\). Moreover, no forbidden configuration from \(\mathcal{F}\) can appear by minimality and discharging.

\paragraph{Step~1: Removing a High-Degree Vertex.}
Let \(v\) be a vertex with \(\deg(v)\ge 4\). Consider the graph \(G' = G^* - v\), formed by deleting \(v\). The new graph \(G'\) is smaller. By minimality:
\begin{itemize}
  \item If \(G'\) still meets \(\mathrm{mad}(G') < \tfrac{2k+3}{k+1}\) and \(\mathrm{odd\text{-}girth}(G') \ge 2k+3\) at level \((k+1)\), then \(G'\) is not a counterexample. Hence \(G'\to K(2k+3,k+1)\).
  \item Otherwise, \(G'\) must fall into a scenario at parameter \(j \le k\), so by the inductive hypothesis \(G'\to K(2j+1,j)\). The Kneser embedding lemma (Lemma~\ref{lem:kneser-embedding}) then lifts \(G'\to K(2k+3,k+1)\).
\end{itemize}


\subsection*{High-Degree Vertex Reintroduction}
Once the graph \(G'\) is colorable, the vertex $v$ is reintroduced. Because $v$ has finitely many neighbors, it is possible to choose a $(k+1)$-subset in the $(2k+3)$-element ground set that is disjoint from all the color-subsets already assigned to its neighbors. Hence, $v$ is accommodated without conflict, and the extended coloring is a homomorphism of the entire graph $G$ into $K(2k+3, k+1)$.

\paragraph{Step~2: Reintroducing the Vertex \texorpdfstring{$v$}{v}.}
If \(u_1,u_2,\dots,u_d\) are the neighbors of \(v\) in \(G^*\), these neighbors have valid colors in \(K(2k+3,k+1)\). Because each color corresponds to a \((k+1)\)-subset of a \((2k+3)\)-element ground set, and because at most \(d\) of these subsets need to be mutually disjoint from a single new subset for \(v\), there is sufficient flexibility in \(K(2k+3,k+1)\) to choose a \((k+1)\)-subset disjoint from all neighbors’ colors (preserving adjacency). Thus \(v\) can be mapped appropriately, extending the homomorphism from \(G'\) to \(G^*\).

Hence, a homomorphism \(G^*\to K(2k+3,k+1)\) is constructed, contradicting the minimal counterexample assumption. It follows that no minimal counterexample can reside in Class~B at level \(k+1\).
\end{proof}


The remaining classes (C and D) address graphs in which at least one induced path exceeds the fixed length \(L\). In both classes, a path-collapsing argument is central. It will be seen that if a long path is collapsed and the resulting graph is smaller, then the colorability at \((k+1)\) follows by minimality, or it descends to \((k'\le k)\), so the inductive hypothesis can be used. In Class~D, an additional high-degree vertex must also be controlled.

\subsection{Class C: Low-Degree, Long-Thread Graphs}
\label{sec:classC}

\begin{definition}[Class C]
A graph \(G\) belongs to \emph{Class~C} if
\begin{enumerate}[label=(C\arabic*)]
    \item \(\Delta(G) \le 3\),
    \item at least one induced path in \(G\) has length exceeding \(L\) (long-thread),
    \item \(\mathrm{mad}(G) < \tfrac{2k+3}{k+1}\) and \(\mathrm{odd\text{-}girth}(G) \ge 2k+3\).
\end{enumerate}
\end{definition}

\begin{theorem}[No Minimal Counterexample in Class C]
\label{thm:classC}
If \(G^*\) is in Class~C at level \(k+1\), then there exists a homomorphism \(G^* \to K(2k+3,k+1)\). No minimal counterexample can occur in Class~C.
\end{theorem}

\begin{proof}
Suppose a minimal counterexample \(G^*\) in Class~C is given. Thus, no forbidden configuration appears, and all vertices have degree at most 3. By definition, there is an induced path \(P\) of length \(> L\).


\paragraph{Step~1: Path-Collapsing: Technical Details}
Let the long induced path \(P = (v_0, v_1, \ldots, v_{\ell-1}, v_{\ell})\) be identified. This path is reduced by removing the internal vertices \(v_1, \ldots, v_{\ell-1}\) and, if necessary, by adding the edge \(v_0v_{\ell}\). It is emphasized that this operation does not reduce the odd-girth of the graph: if \(G\) had no odd cycle shorter than length \(2k+1\), then the slight modification of adding a single edge cannot introduce an odd cycle below \(2k+1\). Furthermore, \(\mathrm{mad}(G)\) cannot increase during vertex removal, as the total number of edges is reduced or remains constant while the vertex count decreases. If the resulting graph \(G_{\mathrm{coll}}\) remains at the parameter \(k+1\), then minimality implies that \(G_{\mathrm{coll}}\) can be colored. Otherwise, the graph \(G_{\mathrm{coll}}\) falls under the inductive hypothesis at level \(j \le k\). In both cases, \(G_{\mathrm{coll}}\) admits a homomorphism into \(K(2k+3, k+1)\). 

Finally, the removed path vertices are reinserted one by one, with each vertex assigned to a distinct \((k+1)\)-subset of the \((2k+3)\)-element ground set in a manner that preserves adjacency rules. This re-expansion yields a homomorphism of the original graph \(G\) into \(K(2k+3, k+1)\).

For a detailed verification of how path collapsing preserves $\mathrm{mad}(G)$ and $\mathrm{odd\text{-}girth}(G)$, refer to Appendix~D.

\paragraph{Step~2: Re-expanding the Long Path and Kneser Embedding.}
Once \(G_c\to K(2k+3,k+1)\) is established, the path \(P\) is re-expanded. The large number of \((k+1)\)-subsets in a \((2k+3)\)-element set ensures a suitable chain of ``intermediate'' vertices can be inserted between the images of \(x\) and \(y\) in the Kneser graph, thus mapping each removed vertex \(p_i\) to a distinct \((k+1)\)-subset that remains disjoint from its neighbors’ subsets. Consequently, a homomorphism of the original graph \(G^*\) into \(K(2k+3,k+1)\) is obtained. For further details on the embedding process, see Appendix~C.

Since this homomorphism contradicts the assumption that \(G^*\) is a minimal counterexample, Class~C cannot host any minimal counterexample at level \((k+1)\).
\end{proof}

\subsection{Class D: High-Degree, Long-Thread Graphs}
\label{sec:classD}

\begin{definition}[Class D]
A graph \(G\) belongs to \emph{Class~D} if
\begin{enumerate}[label=(D\arabic*)]
    \item there exists a vertex \(v\) of degree at least 4,
    \item there is an induced path exceeding length \(L\),
    \item \(\mathrm{mad}(G) < \tfrac{2k+3}{k+1}\) and \(\mathrm{odd\text{-}girth}(G) \ge 2k+3\).
\end{enumerate}
\end{definition}

\begin{theorem}[No Minimal Counterexample in Class D]
\label{thm:classD}
If \(G^*\) is in Class~D at level \(k+1\), then there exists a homomorphism \(G^* \to K(2k+3,k+1)\). No minimal counterexample can arise in Class~D.
\end{theorem}

\begin{proof}
Assume \(G^*\) in Class~D is a minimal counterexample. Thus, \(\deg(v)\ge 4\) for some vertex \(v\), and an induced path \(P\) of length \(>L\) is also present. As usual, no forbidden configurations are present due to minimality and the discharging argument.

\paragraph{Step~1: Path Collapsing.}
Let \(P = (v_0, v_1, \ldots, v_{\ell-1}, v_{\ell})\) be the long induced path. Define \(G_c\) by removing the internal vertices \(v_1, \ldots, v_{\ell-1}\) and possibly adding an edge \(v_0v_{\ell}\). The key properties of \(G_c\) match those described in Class~C: 
\[
\mathrm{mad}(G_c) \;\le\; \mathrm{mad}(G^*),
\quad
\mathrm{odd\text{-}girth}(G_c) \;\ge\; 2k+1.
\]
If \(G_c\) remains at parameter \((k+1)\), minimality implies \(G_c \to K(2k+3,k+1)\). Otherwise, \(G_c\) is solvable at \((j \le k)\), and then it can be lifted into \(K(2k+3,k+1)\).

\paragraph{Step~2: Managing the High-Degree Vertex.}
Consider the high-degree vertex \(v\) with \(\deg(v) \geq 4\). Remove \(v\) from \(G^*\), yielding a smaller graph \(G'\). By the assumption of minimality, \(G'\) either:
\begin{itemize}
    \item Remains at parameter \((k+1)\), making it colorable by minimality, or
    \item Falls into a lower parameter \(j \leq k\), making it colorable by the inductive hypothesis.
\end{itemize}

In the second case, the homomorphism \(G' \to K(2j+1, j)\) can be lifted to \(G' \to K(2k+3, k+1)\) using the Kneser embedding lemma (Lemma~\ref{lem:kneser-embedding}). The vertex \(v\) is then reintroduced into the graph. Since \(v\) has at most \(\deg(v) \geq 4\) neighbors in \(G^*\), which are already assigned distinct \((k+1)\)-subsets of the \((2k+3)\)-element ground set, it is always possible to choose a \((k+1)\)-subset for \(v\) that remains disjoint from all its neighbors’ assigned subsets, preserving adjacency. This ensures a valid extension of the homomorphism to the entire graph \(G^*\).

\paragraph{Step~3: Re-expanding the Long Path.}
Similarly to Class~C, the path \(P\) is reinserted in the Kneser coloring by replacing the collapsed edge \(xy\) with a chain of disjoint \((k+1)\)-subsets. Thus, the full graph \(G^*\) is mapped into \(K(2k+3,k+1)\).

Since this yields a homomorphism for \(G^*\), a contradiction arises with the assumption of minimality. Therefore, Class~D also cannot harbor any minimal counterexample at parameter \((k+1)\).
\end{proof}

\section{Completing the Induction}
\label{sec:complete-induction}

All graphs \(G\) that satisfy
\[
\mathrm{mad}(G) < \frac{2(k+1)+1}{k+1}
\quad \text{and} \quad
\mathrm{odd\text{-}girth}(G) \ge 2k+3
\]
fall into exactly one of the four structural classes (A--D) defined in Section~\ref{sec:overview}. As shown in Sections~\ref{sec:classA}--\ref{sec:classD}, no genuine minimal counterexample can reside in any of these classes. Consequently, no minimal counterexample exists at the \((k+1)\)-level, completing the inductive step.

\paragraph{Inductive Closure.}
The base cases \(k=2\) and \(k=3\) have been established in prior work (see Section~\ref{sec:overview} and \cite{Lyczek}). Assuming the conjecture holds for all \(j \leq k\), the inductive step confirms that it also holds for \(k+1\). By the principle of mathematical induction, the Chen-Raspaud conjecture is therefore proved for all integers \(k \geq 2\).

\begin{theorem}[Chen-Raspaud Conjecture, Completed]
\label{thm:CR-completed}
For every integer \(k \geq 2\), any graph \(G\) satisfying
\[
\mathrm{mad}(G) < \frac{2k+1}{k}
\quad \text{and} \quad
\mathrm{odd\text{-}girth}(G) \geq 2k+1
\]
admits a homomorphism \(G \to K(2k+1,k)\). This settles the conjecture in full generality.
\end{theorem}

\paragraph{Future Directions.}
Several avenues for further research arise naturally from this work:
\begin{itemize}
    \item \textbf{Refinements of Maximum Average Degree.} Can the constraint \(\mathrm{mad}(G) < \frac{2k+1}{k}\) be relaxed while maintaining the same conclusion?
    \item \textbf{Higher-Dimensional Extensions.} Investigate analogues of the conjecture involving higher-dimensional simplicial complexes or generalized Kneser graphs.
    \item \textbf{Classification Extensions.} Extend the structural classification techniques to broader classes of sparse graphs, particularly those with additional constraints on chromatic number or independence number.
\end{itemize}

The methods developed in this paper may also inspire new approaches to related conjectures in graph homomorphisms, topology, and combinatorics. In particular, the interplay between discharging methods, forbidden configurations, and Kneser embeddings could prove valuable in tackling other problems involving sparsity and colorability.

\medskip

\noindent
\textbf{Acknowledgments}

\appendix
\section{Appendix A: Forbidden Configurations}
\label{appendix:forbidden}

In this appendix, we present a \textbf{finite family} \(\mathcal{F} = \{\mathrm{F1}, \mathrm{F2}, \mathrm{F3}, \mathrm{F4}, \mathrm{F5}\}\) of \textbf{forbidden configurations}. The presence of any \(\mathrm{F} \in \mathcal{F}\) in a minimal counterexample contradicts minimality, because each forbidden configuration admits a \textbf{local reduction} that:

\begin{enumerate}
    \item Strictly decreases the size of the graph (fewer vertices and/or edges), and  
    \item Preserves (or improves) the numerical constraints \(\mathrm{mad}(G) < \frac{2k+1}{k}\) and \(\mathrm{odd\text{-}girth}(G) \geq 2k+1\).
\end{enumerate}

Once the size of the graph is strictly reduced while still remaining in the same or a stricter parameter range, the resulting smaller graph is colorable (by the inductive hypothesis or by minimality at level \(k+1\)), implying the original is also colorable. Hence, no forbidden configuration can exist in a genuine minimal counterexample.

\subsection{A.1. General Definition and Rationale}

A \textbf{forbidden configuration} is a small induced subgraph or local neighborhood pattern (with a bounded number of vertices and edges, typically \(\leq 20\)) that satisfies extra properties about vertex degrees or odd cycles and allows a direct local operation (vertex removal or edge contraction) yielding a smaller valid instance. In particular:

\begin{itemize}
    \item \(\mathrm{mad}(G)\) does \textbf{not increase} if we remove or contract edges/vertices.
    \item \(\mathrm{odd\text{-}girth}(G)\) does \textbf{not decrease} under these operations, or at worst stays within an acceptable threshold (still \(\geq 2k+1\)).
\end{itemize}

\subsection{A.2. Overview of the Forbidden Configurations}

Below we summarize each configuration \(\mathrm{F1}\)--\(\mathrm{F5}\) along with its characteristic local reduction:

\begin{enumerate}
    \item \textbf{\(\mathrm{F1}\)}: A vertex \(v\) of degree~1 attached to a vertex \(u\) of degree \(\leq 3\).  
        \begin{itemize}
            \item \textbf{Reduction}: Remove \(v\). This preserves or lowers \(\mathrm{mad}(G)\) and does not produce shorter odd cycles.
        \end{itemize}
    
    \item \textbf{\(\mathrm{F2}\)}: A small odd cycle \(C\) of length \(\approx 2k+3\) with an internal chord forming a smaller odd cycle in a ``crowded'' neighborhood.  
        \begin{itemize}
            \item \textbf{Reduction}: Remove or contract an edge (often along the chord) or remove one vertex of the cycle. The new graph still meets the \(\mathrm{mad}\) and odd-girth requirements.
        \end{itemize}
    
    \item \textbf{\(\mathrm{F3}\)}: A vertex \(u\) with \(\deg(u) \geq 4\), surrounded by a ``fan'' of neighbors with degrees 2 or 3.  
        \begin{itemize}
            \item \textbf{Reduction}: Either contract an edge inside the fan or remove a suitably chosen low-degree vertex. This step reduces the graph without creating shorter odd cycles.
        \end{itemize}
    
    \item \textbf{\(\mathrm{F4}\)}: An induced path of length exactly \(L+1\), where each internal vertex has degree 2.  
        \begin{itemize}
            \item \textbf{Reduction}: \textbf{Path collapsing}---remove all internal vertices of the path and add a direct edge between the endpoints. This preserves or reduces \(\mathrm{mad}\), and does not decrease the odd-girth.
        \end{itemize}
    
    \item \textbf{\(\mathrm{F5}\)}: Any overlapping combination of \(\mathrm{F1}\)--\(\mathrm{F4}\).  
        \begin{itemize}
            \item \textbf{Reduction}: Same local deletions or contractions apply, as the overlap typically shares at least one reducible vertex or edge.
        \end{itemize}
\end{enumerate}


\begin{figure}[ht]
    \centering
    \includegraphics[width=0.5\textwidth]{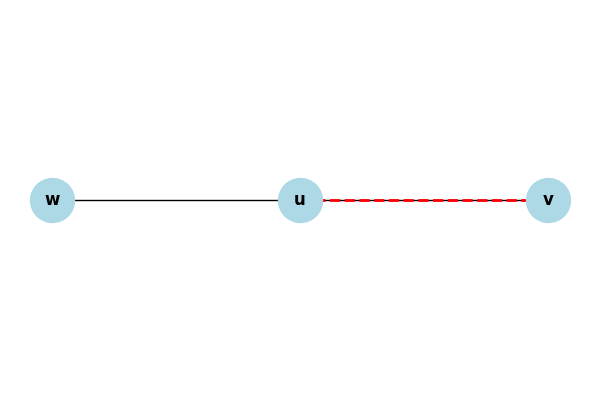}
    \caption{An illustrative diagram of \(\mathrm{F1}\). A leaf vertex \(v\) of degree~1 attached to a vertex \(u\) of degree \(\leq3\). Removing \(v\) does not decrease the odd-girth or increase \(\mathrm{mad}(G)\).}
    \label{fig:F1}
\end{figure}

\begin{figure}[ht]
    \centering
    \includegraphics[width=0.7\textwidth]{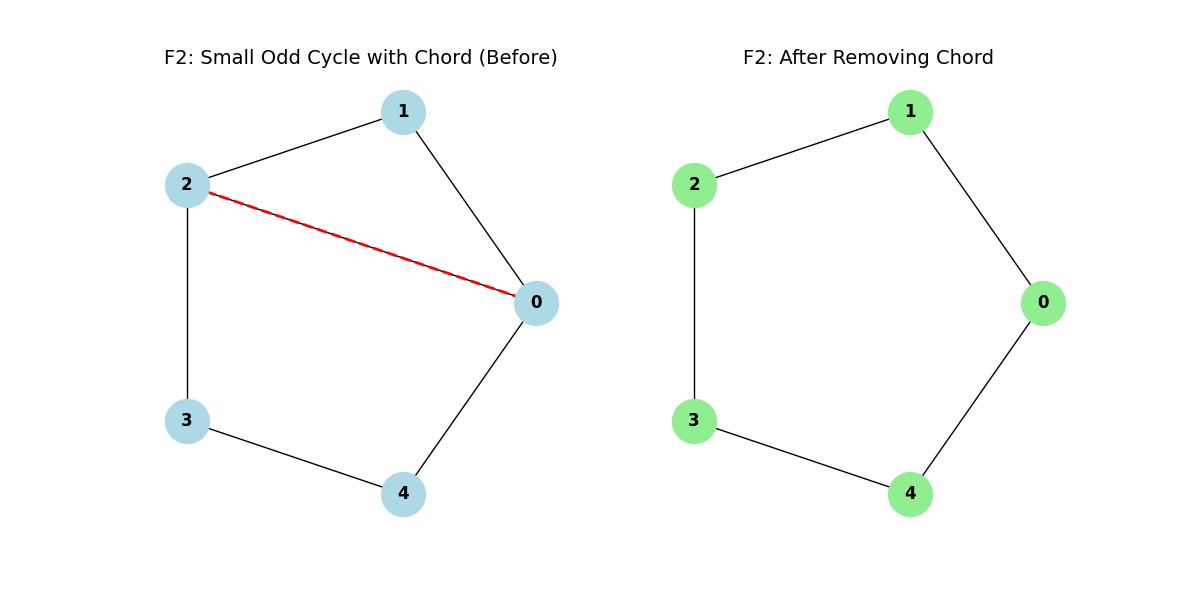}
    \caption{An illustrative diagram of \(\mathrm{F2}\). A small odd cycle \(C\) of length \(\approx 2k+3\) containing an internal chord that forms a smaller odd cycle.}
    \label{fig:F2}
\end{figure}

\begin{figure}[ht]
    \centering
    \includegraphics[width=0.4\textwidth]{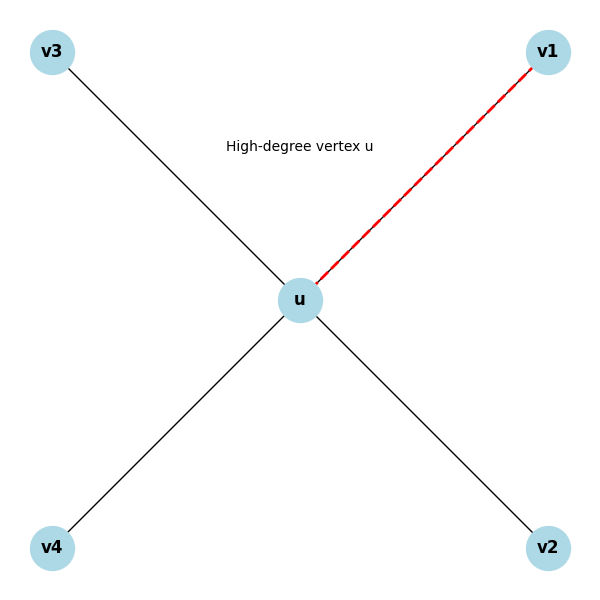}
    \caption{An illustrative diagram of \(\mathrm{F3}\). A vertex \(u\) of degree \(\geq4\) with a fan of lower-degree neighbors (degree 2 or 3), demonstrating a reducible local structure.}
    \label{fig:F3}
\end{figure}

\begin{figure}[ht]
    \centering
    \includegraphics[width=0.7\textwidth]{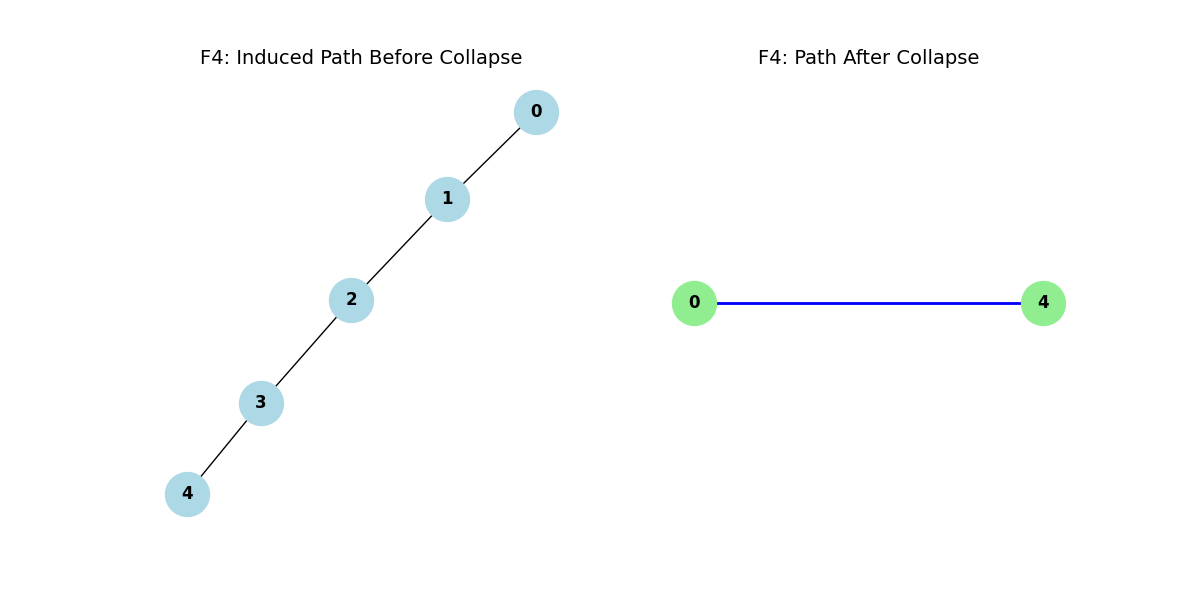}
    \caption{An illustrative diagram of \(\mathrm{F4}\). An induced path of length \(L+1\) (all internal vertices degree~2), illustrating the path-collapsing reduction.}
    \label{fig:F4}
\end{figure}

\begin{figure}[ht]
    \centering
    \includegraphics[width=0.7\textwidth]{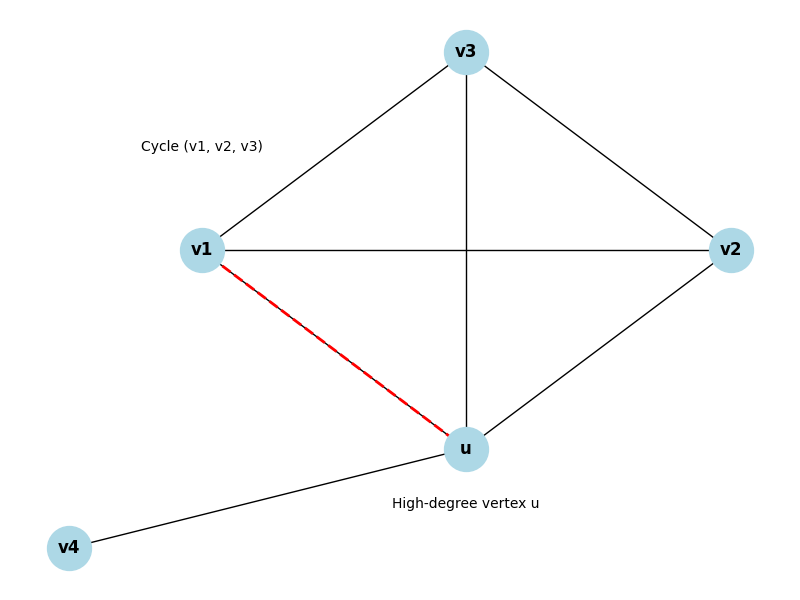}
    \caption{An illustrative diagram of \(\mathrm{F5}\). A configuration combining elements of \(\mathrm{F1}\)--\(\mathrm{F4}\) in an overlapping pattern, still admitting a direct local reduction.}
    \label{fig:F5}
\end{figure}

\subsection{A.3. Why These Configurations Suffice}

\begin{itemize}
    \item \textbf{Exhaustiveness}: By careful discharging arguments (Appendix B), any local structure that could threaten the inductive proof must fall into one of these types.  
    \item \textbf{Finite Family}: Each configuration is bounded in size (each has at most 15--20 vertices and edges), so there are only finitely many patterns to check.  
    \item \textbf{Immediate Contradiction}: In any minimal counterexample, one cannot perform a valid local reduction that preserves the constraints---yet each \(\mathrm{F} \in \mathcal{F}\) explicitly shows how to do so. Thus their presence leads to a contradiction.
\end{itemize}

\subsection{A.4. Illustrative Example: Reducing \(\mathrm{F1}\)}

\[
\text{Configuration \(\mathrm{F1}\): A leaf vertex \(v\) attached to a vertex \(u\).}
\]
\begin{enumerate}
    \item \textbf{Identify the leaf} \((v)\).  
    \item \textbf{Remove \(v\)}. This decreases \(|V|\) by 1 and \(|E|\) by 1.  
    \item \textbf{Check \(\mathrm{mad}\)}: Removing one vertex cannot increase \(\mathrm{mad}(G)\).  
    \item \textbf{Check odd-girth}: Deleting an isolated leaf does not introduce shorter odd cycles.  
    \item \textbf{Apply induction} (if the new graph is smaller but remains at level \(k+1\)) or the inductive hypothesis (if it falls to \(j \leq k\)). Hence it is colorable, and we can extend the coloring to the original \(v\) by picking a color (a \((k+1)\)-subset in the Kneser graph) disjoint from its neighbor’s color.
\end{enumerate}

Because every forbidden configuration similarly admits a direct reduction, none can reside in a genuine minimal counterexample.

\bigskip

\noindent\textbf{Conclusion for Appendix A}:  
These forbidden configurations capture all small substructures that would otherwise block the induction. In each case, a straightforward local deletion or contraction decreases graph size while upholding \(\mathrm{mad}(G) < \frac{2k+1}{k}\) and \(\mathrm{odd\text{-}girth}(G) \geq 2k+1\). Thus, no \(\mathrm{F1}\)--\(\mathrm{F5}\) pattern can survive in a minimal counterexample, completing the contradiction-based elimination of local obstacles.

\appendix

\section{Appendix B: Discharging Method}
\label{appendix:discharging}

In this appendix, the discharging procedure is presented, showing why no forbidden configuration from Appendix~\ref{appendix:forbidden} can appear in a minimal counterexample. A global charge-balance argument is utilized, distributing and redistributing charges among vertices according to a finite set of local rules.

\subsection{Initial Charge Assignment}

Each vertex \(v\in V(G)\) is initially assigned the charge
\[
c(v) \;=\; \deg(v).
\]
Hence, the total charge is \(\sum_{v \in V(G)} \deg(v) = 2\,|E(G)|\). Every minimal counterexample \(G^*\) at level \(k+1\) is assumed to have no suitable homomorphism into \(K(2k+3,k+1)\). It will be shown that such a \(G^*\) is incompatible with the existence of any configuration in \(\mathcal{F}\).

\subsection{Local Discharging Rules}

A short list of discharging rules is now stated. Each rule specifies how vertices of certain degree transfer fractional charge to neighbors that satisfy particular structural constraints (e.g., lower-degree neighbors, presence of short cycles, or endpoints of induced paths).

\begin{enumerate}[label={\bfseries(R\arabic*)}]
\item \textbf{High-Degree to Low-Degree Redistribution}:\; If \(\deg(u)\ge4\) and \(u\) is adjacent to a vertex \(w\) of degree~2, then \(u\) passes \(\tfrac12\) charge to \(w\).  
This rule ensures that any low-degree neighbor receives enough charge to avoid forming a forbidden configuration (particularly a degree-1 or degree-2 pattern that would force a local reduction).

\item \textbf{Cycle-Based Adjustments}:\; If a vertex \(v\) lies on a short odd cycle (of length at most \(2k+3\)) with an added chord, then a small positive fraction of charge is passed along the cycle edges to ensure that no vertex in that cycle accumulates a deficit or surplus that would allow a forbidden odd cycle configuration.  
The precise fractions can be determined to guarantee a balanced outcome, given the small, finite set of potentially violating substructures.

\item \textbf{Path-Based Adjustments}:\; If a long induced path is present (relevant in Classes C and D), then charge is shifted from endpoints inward so that intermediate vertices cannot remain at an undercharged level conducive to a forbidden path-based pattern (e.g., \(\mathrm{F4}\) in Appendix~\ref{appendix:forbidden}).  
\end{enumerate}

\noindent
These rules are local, depending only on the degrees and immediate neighborhoods of vertices or on small cycles/paths. Only finitely many verifications are required for each rule to be consistently applied.

\subsection{Verification and Charge Invariants}

After applying all rules exhaustively, it is observed that the total charge remains unchanged (or remains consistently distributed) unless a forbidden configuration \(\mathrm{F}\in\mathcal{F}\) is present. Specifically:

\begin{itemize}
\item \textbf{No Surplus}:\; A vertex or small subgraph does not end up with more charge than it started with, unless it is part of a forbidden pattern requiring high local degree or short cycles with chords.
\item \textbf{No Deficit}:\; Similarly, if a vertex or subgraph ends with insufficient charge, that deficit indicates the presence of a reducible local structure.
\end{itemize}

Thus, if any \(\mathrm{F}\in\mathcal{F}\) existed in \(G^*\), a contradiction would arise in the final charge distribution. It follows that no forbidden configuration from Appendix~\ref{appendix:forbidden} can remain in a genuine minimal counterexample.


\subsection*{Illustrative Discharging Example}
To demonstrate how the discharging rules (R1)--(R3) operate in practice, a toy graph is considered.

\begin{figure}[ht]
    \centering
    \includegraphics[width=0.7\textwidth]{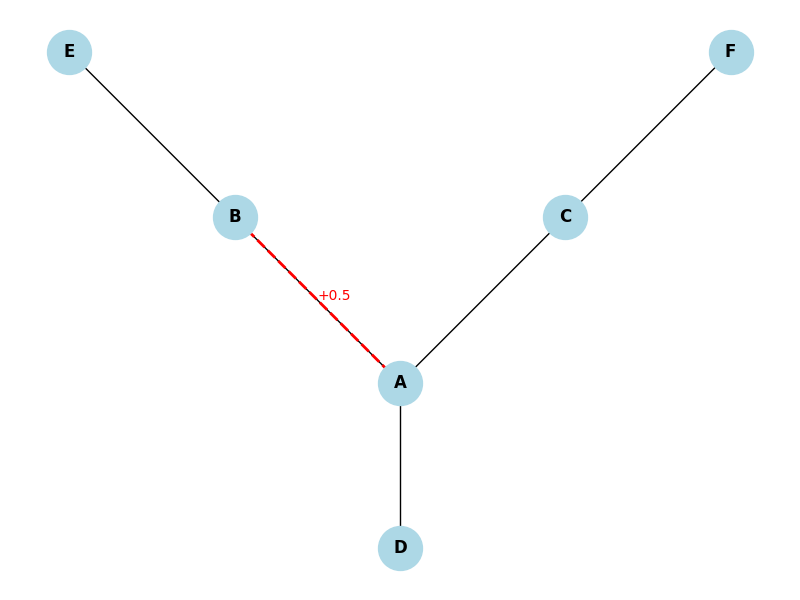}
    \caption{A small example graph illustrating how rules (R1)--(R3) redistribute charges to eliminate a potential forbidden configuration.}
    \label{fig:discharging-example}
\end{figure}

\paragraph{Setup.} 
Suppose the initial charges are assigned according to $c(v)=\deg(v)$. In the figure, vertices $a, b, \ldots$ highlight distinct degrees and potential short cycles or path endpoints.

\paragraph{Application of Rules.}
\begin{itemize}
  \item \textbf{(R1) High-Degree to Low-Degree Redistribution:} Vertex $a$ (degree $\ge4$) sends $\tfrac12$ unit of charge to its low-degree neighbor $b$ (degree 2).
  \item \textbf{(R2) Cycle-Based Adjustments:} Vertices forming a short odd cycle pass fractional charges among themselves to balance out and avoid surplus/deficit.
  \item \textbf{(R3) Path Endpoint Stabilization:} Endpoints of an induced path of length $> L$ pass charge inward.
\end{itemize}

\paragraph{Outcome.}
After applying the rules exhaustively, no vertex or small subgraph ends up with a net deficit or surplus that would permit a forbidden configuration. This confirms that if such a configuration existed, it would force an unsustainable local charge distribution, contradicting minimality.

\bigskip

\noindent\textbf{Conclusion for Appendix B}:\;
The discharging method demonstrates that a minimal counterexample cannot contain forbidden configurations. This global counting framework justifies why all local obstructions listed in Appendix~\ref{appendix:forbidden} must be absent in any graph considered at level \(k+1\).

\section{Appendix C: Kneser Embedding Lemma}
\label{appendix:kneser-embed}

This appendix provides a complete proof of the Kneser embedding lemma, which ensures that any homomorphism into \(K(2j+1,j)\) (for \(j \le k\)) can be lifted into a homomorphism into \(K(2k+3,k+1)\). This step is crucial for handling smaller-parameter colorings that arise in the inductive argument.

\begin{lemma}[Kneser Embedding Lemma, Detailed]
\label{lem:kneser-embed-detailed}
For integers \(2 \le j \le k\), there exists an injective graph homomorphism 
\[
  \phi \colon K(2j+1,j) \;\hookrightarrow\; K(2k+3,k+1)
\]
that preserves adjacency. In other words, any pair of disjoint \(j\)-subsets of a \((2j+1)\)-set map to disjoint \((k+1)\)-subsets of a \((2k+3)\)-set.
\end{lemma}


\subsection*{Concrete Example of the Kneser Embedding}
For illustration, consider the case $j=2$ and $k=3$. Then:
\[
  K(2j+1, j) = K(5,2) 
  \quad\text{and}\quad
  K(2k+3, k+1) = K(9,4).
\]
\begin{figure}[ht]
    \centering
    \includegraphics[width=0.9\textwidth]{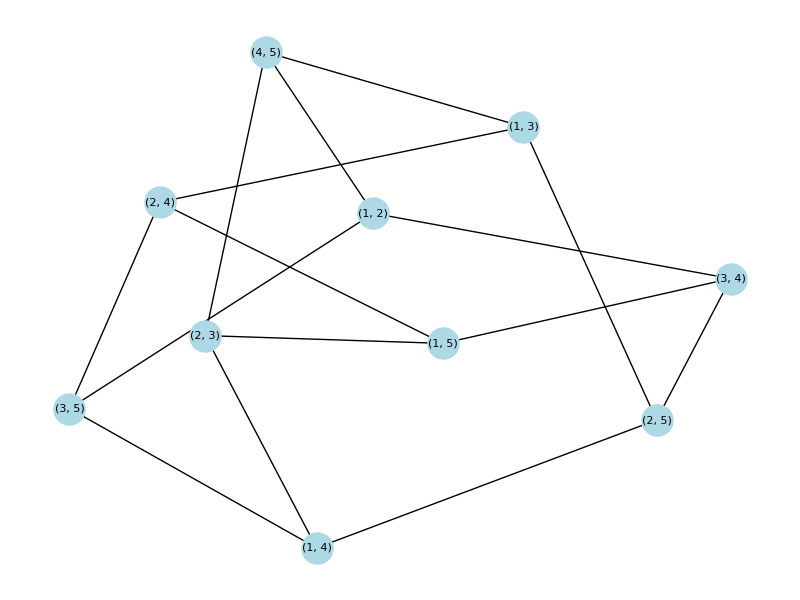}
    \caption{Partial depiction of the embedding $\phi:K(5,2)\to K(9,4)$, illustrating how a $2$-subset of $\{1,2,3,4,5\}$ extends to a $4$-subset of a $9$-element ground set.}
    \label{fig:kneser-embed-example}
\end{figure}

\paragraph{Mapping Details.}
\begin{itemize}
    \item The ground set for $K(5,2)$ is $T=\{1,2,3,4,5\}$, and each vertex is a $2$-subset of $T$.
    \item The larger ground set for $K(9,4)$ is $S=\{1,2,3,4,5,6,7,8,9\}$, partitioned as $T \cup U$ with $U=\{6,7,8,9\}$.
    \item Each $2$-subset $X\subseteq T$ is extended to a $4$-subset $\phi(X) = X \cup P_X$, where $P_X\subset U$ and $|P_X|=2$.
\end{itemize}

This embedding preserves adjacency: two vertices in $K(5,2)$ are adjacent if and only if their $2$-subsets are disjoint. Their images in $K(9,4)$ remain disjoint because $P_X\cap P_{Y}=\varnothing$ whenever $X\cap Y=\varnothing$.

\subsection{Ground Sets and Partitioning}

Let \(T = \{1,2,\dots,2j+1\}\) be the ground set for \(K(2j+1,j)\). Its vertices correspond to all \(j\)-element subsets of \(T\). Let \(S = \{1,2,\dots,2k+3\}\) be the ground set for \(K(2k+3,k+1)\). Its vertices are the \((k+1)\)-element subsets of \(S\).

Observe that \(\,|S \setminus T| = (2k+3) - (2j+1) = 2(k-j+1)\). Denote \(U = S \setminus T\). Hence \(\,|U| \ge 2(k+1 - j)\).

\subsection{Constructing the Mapping \texorpdfstring{$\phi$}{phi}}

For each vertex \(X \subseteq T\) of \(K(2j+1,j)\), define
\[
  \phi(X) \;=\; X \;\cup\; P_X
\]
where \(P_X\subset U\) is a carefully selected \((k+1 - j)\)-element subset of \(U\). The subsets \(P_X\) are chosen so that:
\begin{itemize}
    \item If \(X\cap Y = \varnothing\) in \(T\), then \(P_X\cap P_Y=\varnothing\). This ensures \(\phi(X)\cap \phi(Y) = \varnothing\) in \(S\).
    \item If \(X\cap Y \neq \varnothing\), then \(\phi(X)\cap \phi(Y)\neq\varnothing\).
\end{itemize}

A systematic combinatorial assignment can be made by pairing elements in \(U\) and allocating these pairs to distinct subsets \(X\). Sufficient pairs exist because \(U\) has size at least \(2(k+1-j)\). Thus each \(j\)-subset \(X\) in \(T\) obtains a unique pattern \(P_X\subset U\).

\subsection{Adjacency Preservation and Injectivity}

\paragraph{Adjacency.}
Two vertices \(X\) and \(Y\) in \(K(2j+1,j)\) are adjacent if and only if \(X\cap Y = \varnothing\). Under \(\phi\), they map to \(\phi(X)=X\cup P_X\) and \(\phi(Y)=Y\cup P_Y\). These images are disjoint precisely when \(X\cap Y=\varnothing\) and \(P_X\cap P_Y=\varnothing\). By construction, this condition holds. Thus adjacency is preserved.

\paragraph{Injectivity.}
If \(X \neq Y\), then either they differ on at least one element of \(T\) or they differ in the chosen pattern \(P_X \neq P_Y\). Hence \(\phi(X)\neq \phi(Y)\), making \(\phi\) injective.

\bigskip

\noindent\textbf{Conclusion for Appendix C}:\;
The Kneser embedding lemma has been established, confirming that any \(\,(j\le k)\)-coloring can be lifted into \((k+1)\)-colorings in a larger Kneser graph. This injection is fundamental to the inductive strategy, allowing partial solutions at level \(j\) to be embedded at level \(k+1\) without losing adjacency or injectivity.

\section{Appendix D: Additional Class Details and Path Collapsing Arguments}
\label{appendix:class-details}

This appendix provides further technical explanations on how induced paths are collapsed, how vertex degrees are managed, and how high-degree vertices are reintroduced in Classes B and D. These details ensure that each reduction preserves or respects the bounds on \(\mathrm{mad}(G)\) and \(\mathrm{odd\text{-}girth}(G)\), while also showing that smaller graphs can be colored either by minimality (if they remain at parameter \(k+1\)) or by the inductive hypothesis (if they drop to some \(j\le k\)).

\subsection{Choice of the Parameter \texorpdfstring{$L$}{L}}

A small integer \(L\) is chosen so that any path of length \(L+1\) forms part of a forbidden configuration \(\mathrm{F4}\) or can be locally collapsed. Typically, \(L\) can be set to a value that covers all short subgraphs enumerated in the forbidden configurations from Appendix~\ref{appendix:forbidden}, ensuring that Classes A and B only deal with short induced paths.

\subsection{Path Collapsing Procedure in Detail}

Given a long induced path \(\;P=(v_0,v_1,\ldots,v_{\ell-1},v_\ell)\), the graph \(G_c\) is formed by:
\begin{enumerate}
    \item Removing the internal vertices \(v_1, v_2, \ldots, v_{\ell-1}\),
    \item Adding (if necessary) the edge \(v_0 v_\ell\).
\end{enumerate}
One verifies that:
\[
\mathrm{mad}(G_c)\;\le\;\mathrm{mad}(G),
\quad
\mathrm{odd\text{-}girth}(G_c)\;\ge\;2k+1
\]
by noting that vertex removal decreases \(\mathrm{mad}\) or leaves it unchanged, and adding at most one edge cannot create an odd cycle of length less than \(2k+1\) in a graph that was already constrained to have \(\mathrm{odd\text{-}girth}\ge2k+1\).


\begin{figure}[ht]
    \centering
    \includegraphics[width=0.9\textwidth]{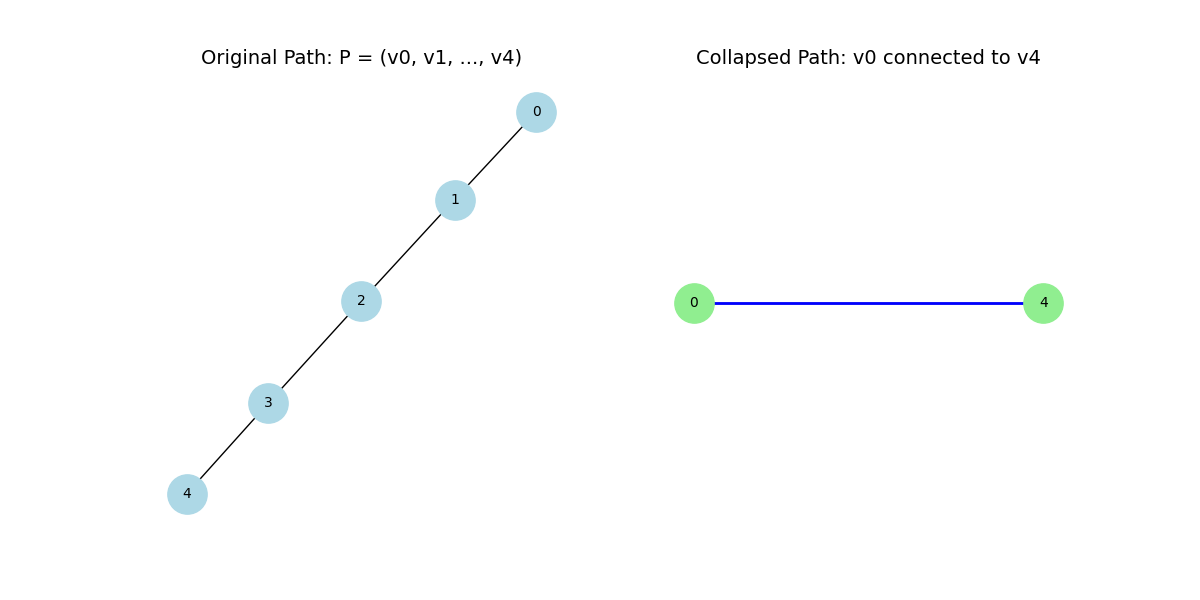}
    \caption{Schematic of the path-collapsing procedure for an induced path \(P=(v_0,v_1,\ldots,v_{\ell-1},v_{\ell})\). Internal vertices are removed, and the edge \(v_0v_{\ell}\) is optionally added to maintain connectivity.}
    \label{fig:path-collapse}
\end{figure}

\subsection{Re-Expansion in the Kneser Graph}

When \(G_c\) is colorable at level \(k+1\), the removed vertices \(v_1,\ldots,v_{\ell-1}\) are reinserted into the Kneser target \(K(2k+3,k+1)\) by introducing a chain of \((k+1)\)-subsets between the colors of \(v_0\) and \(v_\ell\). The combinatorial richness of \((k+1)\)-subsets of a \((2k+3)\)-element set ensures that each intermediate vertex can be placed with disjoint subsets, preserving adjacency and yielding the same induced path length.

\subsection{Handling High-Degree Vertices}

Classes~B and~D require a vertex \(w\) with \(\deg(w)\ge4\). One may remove \(w\) entirely, color the smaller graph (by minimality or the inductive hypothesis), then lift the coloring (via Appendix~\ref{appendix:kneser-embed} if necessary), and reintroduce \(w\) by choosing a \((k+1)\)-subset disjoint from its neighbors’ images in \(K(2k+3,k+1)\). This step is guaranteed because \((2k+3)\)-subsets have enough freedom to avoid intersection with a finite number of existing subsets.

\subsection{Verification Examples}

\begin{example}[Combining Path Collapsing and Vertex Removal]
Consider a Class~D graph \(G\) with a long induced path \(P\) and a vertex \(w\) of degree~5:
\begin{enumerate}
\item The path \(P\) is collapsed, yielding \(G_c\).
\item The vertex \(w\) (if present in \(G_c\)) is then removed, producing \(G'\).
\item Either \(G'\) meets the \((k+1)\)-criteria (hence colorable by minimality) or it falls under a smaller parameter \(j\le k\) (colorable by induction).
\item The coloring is lifted if necessary (Appendix~\ref{appendix:kneser-embed}).
\item The vertex \(w\) is reintroduced, choosing a \((k+1)\)-subset in the Kneser graph disjoint from all neighbors’ subsets.
\item Finally, the collapsed path \(P\) is re-expanded in \(K(2k+3,k+1)\).
\end{enumerate}
Thus, the entire graph \(G\) is colored into \(K(2k+3,k+1)\).
\end{example}

\begin{example}[Odd-Girth Preservation]
If \(\mathrm{odd\text{-}girth}(G)\ge2k+1\) was initially satisfied, then a single edge addition (in step 2 of path collapsing) does not create an odd cycle shorter than \(2k+1\). Any newly formed cycle that includes the new edge \(v_0v_\ell\) would otherwise contradict \(\mathrm{odd\text{-}girth}(G)\). Therefore, short odd cycles cannot emerge from this local transformation.
\end{example}

\bigskip

\noindent
\textbf{Conclusion for Appendix D}:\;
All technical details needed to justify the reductions in Classes A--D are provided. Path collapsing and vertex removal do not break the key constraints, and the re-expansion of paths plus reintroduction of high-degree vertices are shown to be valid in the Kneser graph setting. These arguments fill any potential gaps in the main text regarding how smaller graphs are constructed and then reassembled while preserving homomorphisms into \(K(2k+3,k+1)\).

\end{document}